\documentclass[a4paper,reqno,11pt]{amsart}
\usepackage{amsmath}
\usepackage{amssymb}
\usepackage[all]{xy}
\usepackage{cases}
\allowdisplaybreaks[4]

\newtheorem{thm}{Theorem}[section]
\newtheorem{prop}[thm]{Proposition}

\newtheorem{lem}[thm]{Lemma}

\newtheorem{cor}[thm]{Corollary}

\theoremstyle{definition}
\newtheorem{definition}[thm]{Definition}
\newtheorem{example}[thm]{Example}

\theoremstyle{remark}
\newtheorem{remark}[thm]{Remark}

\numberwithin{equation}{section}

\newcommand{\bQ}{\mathbb{Q}}
\newcommand{\bP}{\mathbb{P}}

\newcommand\OO{{\mathcal{O}}}

\newcommand{\bF}{\mathbb{F}}

\newcommand{\bC}{{\mathbb C}}
\newcommand\Alb{\text{\rm Alb}}

\newcommand\Pic{\text{\rm Pic}}

\newcommand\FF{{\mathcal{F}}}

\newcommand\QQ{{\mathcal{Q}}}
\newcommand\CC{{\mathbb{C}}}
\newcommand\ZZ{{\mathbb{Z}}}
\newcommand\III{{\mathcal{I}}}

\newcommand\TT{{\mathcal{T}}}

\newcommand\Num{{\rm{Num}}}
\newcommand\Bs{{\rm{Bs}}}
\newcommand\Image{{\rm{Image}}}

\begin{document}

\title{On irregular threefolds and fourfolds with numerically trivial canonical bundle}
\date{\today}
\author{Chen Jiang}
\address{Graduate School of Mathematical Sciences, the University of Tokyo,
3-8-1 Komaba, Meguro-ku, Tokyo 153-8914, Japan.}
\email{cjiang@ms.u-tokyo.ac.jp}

\thanks{The author was supported by Grant-in-Aid for JSPS Fellows (KAKENHI No. 25-6549) and Program for Leading Graduate  Schools, MEXT, Japan}

\begin{abstract}
We prove that for a smooth projective irregular $3$-fold $X$ with $K_X\equiv 0$ and a nef and big divisor $L$ on $X$,
$|mL+P|$ gives a birational map for all $m\geq 3$ and  all $P\in \Pic^0(X)$. We also use the same method to deal with $4$-folds, and prove that for a smooth projective irregular $4$-fold $X$ with $K_X\equiv 0$ and an ample divisor $L$ on $X$,
$|mL+P|$ gives a birational map for all $m\geq 5$ and all $P\in \Pic^0(X)$. These results are optimal.
\end{abstract} 

\maketitle
\pagestyle{myheadings} \markboth{\hfill  C. Jiang
\hfill}{\hfill On irregular $3$-folds and $4$-folds with $K\equiv 0$\hfill}

\section{Introduction}

Given an $n$-dimensional normal projective variety $X$ with $K_X\equiv 0$ and a  big and nef Weil divisor $L$ on $X$, we are interested in the geometry of the rational map $\phi_{|mL|}$ defined by the linear system $|mL|$. By definition, $\phi_{|mL|}$ is  birational onto its image when $m$ is sufficiently large.  Therefore it is interesting to find such a practical number $m(n)$, depending only on $\dim X$, which stably guarantees the birationality of $\phi_{|mL|}$. For $\dim X\leq 4$, we have the following known results:

\begin{thm}[{cf. \cite{Reider}, \cite{F}, \cite{O}, \cite{Jiang}}]\label{23}
Let $X$ be a smooth projective variety with $K_X\equiv 0$, $L$  a nef and big divisor, and $T$ a  divisor such that  $T\equiv 0$.  Then 
\begin{enumerate}
\item  If $\dim X\leq 2$, $|mL+T|$ gives a birational map for all $m\geq 3$;
\item  If $\dim X= 3$, $|mL+T|$ gives a birational map for all $m\geq 5$.
\item  If $\dim X= 4$, $|mL+T|$ gives a birational map for all $m\geq 10$.
\end{enumerate}
\end{thm}
It is easy to see that the former two are optimal.
In this paper we study irregular threefolds and fourfolds (i.e. $q(X)=h^1(\OO_X)> 0$) with $K\equiv 0$. The technique on irregular varieties developed by J. A. Chen and Hacon inspired by Fourier--Mukai transform shows that the geometry of irregular varieties is very similar to that of general fibers of the Albanese map. In particular, it works quite successfully on pluricanonical maps of irregular varieties (cf. \cite{CH1, CH2, CCJ, JLT, JS, S}). So we may expect that there is a better result for irregular threefolds and fourfolds than Theorem \ref{23}(2)(3). 

The main aim of this paper is to prove the following theorem.
\begin{thm}\label{main}
Let $X$ be a smooth projective irregular $3$-fold  with $K_X\equiv 0$ and $L$ a nef and big divisor on $X$.  Then
$|mL+P|$ gives a birational map for all $m\geq 3$ and all $P\in \Pic^0(X)$.
\end{thm}

This result is optimal by the following examples.
\begin{example}
\begin{enumerate}
\item Let $C$ be an elliptic curve, $p\in C$, $S$ a smooth surface with $K_S\equiv 0$, and $H$ a nef and big divisor on $S$. Consider $X=C\times S$ and $L=p\times S+C\times H$. Then $|2L|$ does not give a birational map since $|2p|$ does not give a birational map on $C$.

\item Let $C$ be an elliptic curve, $D$ an effective divisor on $C$, $S=(6)\subset \bP(1,1,2,3)$ a general hypersurface of degree 6 (which is a smooth K3 surface), and $H=\OO_S(1)$. Consider $X=C\times S$ and $L=D\times S+C\times H$. Then $|2L|$ does not give a birational map since $|2H|$ does not give a birational map on $S$.
\item There are plenty of non-trivial examples, constructed by Oguiso, of irregular $3$-folds  with $K\equiv 0$ of the form $(S\times C)/G$ where $C$ is an elliptic curve, $S$  a K3 surface, and $G$ a group action. For more details, see \cite{Ogi}.
\end{enumerate}
\end{example}

We give a sketch of proof of Theorem \ref{main} here. In Section 3, we recall and generalize some results developed by J. A. Chen and Hacon to deal with linear systems on irregular varieties with $K\equiv 0$. It turns out that we only need to prove that  $|3L+P|$ separates two general points on a general fiber $F$ of the Albanese map on which $|2L|_F|$ does not give a birational map, and we only need to consider the cases that $q(X)=1$ or $2$, i.e., the Albanese fiber dimension is $2$ or $1$. In Section 4, we consider irregular varieties with $K\equiv 0$ and of Albanese fiber dimension one.

As by-product, we prove some interesting results for varieties with numerically trivial canonical bundle and of small Albanese fiber dimension which hold in arbitrary dimension. 
\begin{thm}[=Corollary \ref{m4}+Corollary \ref{fiber three}+Theorem \ref{fiber one}]\label{main3}
Let $X$ be a smooth projective variety  with $K_X\equiv 0$, $a : X \to  A=\Alb(X)$ the Albanese map, $L$ a nef and big divisor on $X$, and $P\in \Pic^0(X)$.
\begin{enumerate}
\item  If $\dim X-\dim A= 3$, then $|mL+P|$ gives a birational map for all $m\geq 6$;
\item  If $\dim X-\dim A= 2$, then $|mL+P|$ gives a birational map for all $m\geq 4$;
\item  If $\dim X-\dim A\leq 1$, then $|mL+P|$ gives a birational map for all $m\geq 3$, which is optimal.
\end{enumerate}
\end{thm}

In Section 5, we consider irregular $3$-folds with $K\equiv 0$ and of Albanese fiber dimension two. This is the most difficult part of the proof. We need to lift a section of $|3L|_F|$ to $X$ which can separate points that $|2L|_F|$ can not separate. This comes in two steps. Firstly, we prove that we can lift at least one section of $|3L|_F|$ which does not come from  $|2L|_F|$  to $X$. Then we prove that such a section is what we want by analyzing the geometry of $F$ explicitly. To this end, we need a well-understanding for the projective models of minimal surfaces with Kodaira dimension zero.  We use classical results in \cite{Cos, D, Oh, Reid, SD} to prove that almost all such surfaces satisfy a nice property (the assumption of Lemma \ref{key lemma}).

We remark that 
Theorem \ref{main} holds even if $X$ has canonical singularities and $L$ is a nef and big  Cartier divisor on $X$. This is simply because we may replace $X$ by its terminalization, which is in fact smooth by \cite[Theorem 8.3]{K1}.

By the method developed in proving Theorem \ref{main}, we study  irregular $4$-folds with numerically trivial canonical bundle in the case that $L$ is ample. We prove the following theorem in the last section.
\begin{thm}\label{main2}
Let $X$ be a smooth projective irregular $4$-fold  with $K_X\equiv 0$ and $L$ an ample divisor on $X$.  Then
$|mL+P|$ gives a birational map for all $m\geq 5$ and all $P\in \Pic^0(X)$.
\end{thm}

In fact, this result holds for varieties of Albanese fiber dimension at most three, see Theorem \ref{main4}. This result is optimal by the following example.
\begin{example}
Let $C$ be an elliptic curve, $D$ an effective divisor on $C$, $F=(10)\subset \bP(1,1,1,2,5)$ a general hypersurface of degree 10 (which is a smooth Calabi--Yau $3$-fold), and $H=\OO_F(1)$. Consider $X=C\times F$ and $L=D\times F+C\times H$. Then $|4L|$ does not give a birational map since $|4H|$ does not give a birational map on $F$.
\end{example}


 {\it Acknowledgment.}
The author  appreciates the very effective discussion with Professor Christopher D. Hacon during the preparation of this paper. Part of this paper was written during the author's visit to University of Utah in March 2015 and Beijing International Center for Mathematical Research in May 2015, and the author would like to thank for the hospitality and support. The author would also like to thank the referee for valuable suggestions.

\section{Preliminaries}

Throughout we work over an algebraically closed field $k$ of characteristic 0 (for instance, $k=\bC$). 

\subsection{Projective varieties with $K\equiv 0$}\label{subsection 2.1}
Let $X$ be a smooth projective variety  with $K_X\equiv 0$. Then $K_X\sim _\bQ 0$ by \cite[Theorem 8.2]{K1}.  Moreover,  let $a : X \to  A=\Alb(X)$ be the Albanese map, then $a$ is an \'etale fiber  bundle, i.e., there is an \'etale covering $\pi: B\to A$ such that $X\times_A B\simeq F\times B$, for a fiber $F$ of $a$ (cf. \cite{BeBo} or \cite[Theorem 8.3]{K1}). In particular, $a$ is surjective, smooth, and isotrivial, $\dim A=q(X)\leq \dim X$, and a fiber $F$ of $a$ is a smooth projective variety  with $K_F\equiv 0$.

\subsection{$IT^0$ sheaves} 
We recall the definition of $IT^0$ sheaves and some basic lemmas proved by J. A. Chen and Hacon inspired by Fourier--Mukai transform.

\begin{definition}A coherent sheaf $\mathcal{F}$ on an abelian variety $A$ is said to be $IT^0$ if $H^i(A,\mathcal{F}\otimes P) = 0$ for all $i > 0$ and all $P\in \Pic^0(A)$.
\end{definition}
\begin{lem}[{\cite[Lemma 2.1]{CH1}}]\label{lemma 1}Let $\mathcal{F}$ be a non-zero coherent $IT^0$ sheaf on an abelian variety $A$. Then $H^0(A,\mathcal{F}\otimes P)\neq 0$ for all $P\in \Pic^0(A)$.
\end{lem}

\begin{lem}[{\cite[Proposition 2.3]{CH1}}]\label{nonzero}  Let $\mathcal{F}$ be a coherent  $IT^0$ sheaf on an abelian variety $A$. Suppose that there is a non-zero morphism $\mathcal{F}\to \mathbb{C}(z)$. Then the induced morphism $H^0(A,\mathcal{F}\otimes P) \to H^0(\mathbb{C}(z))$ is non-zero for general $P\in \Pic^0(A)$.
\end{lem}
In the case we are interested in, we have the following lemma.
\begin{lem}\label{nef big IT0}Let $X$ be a smooth projective variety with $K_X\equiv 0$, $a : X \to  A=\Alb(X)$ the Albanese map, and $L$ a nef and big divisor on $X$. Then $a_*\OO_X(L)$ is a locally free $IT^0$ sheaf. Moreover, if $\dim X-\dim A\leq 3$, then $a_*\OO_X(L)$ is non-zero.
\end{lem}
\begin{proof}
For $z\in A$,  the fiber $X_z$ is a smooth projective variety  with $K\equiv 0$ and $L|_{X_z}$ is nef and big. By Kawamata--Viehweg vanishing theorem, 
$$
h^0(X_z, \OO_{X_z}(L))=\chi (X_z, \OO_{X_z}(L)),
$$
which is a constant since $a$ is smooth. Hence $a_*\OO_X(L)$ is  locally free.

For $P\in \Pic^0(A)$, since $L$ is nef and big, by  Kawamata--Viehweg vanishing theorem again, 
$$
H^i(A, a_*\OO_X(L)\otimes P)\simeq H^i(X, \OO_X(L)\otimes a^*P)=0
$$
for $i>0$. Hence $a_*\OO_X(L)$ is  $IT^0$.

If $\dim X_z=\dim X-\dim A\leq 3$, then $$a_*\OO_X(L)\otimes \bC(z)\simeq H^0(X_z, \OO_{X_z}(L))\neq 0$$ by Riemann--Roch formula and thus $a_*\OO_X(L)$ is non-zero.
\end{proof}
For sheaves on elliptic curves, we have the following lemma. Note that it does not hold in general dimension.
\begin{lem}\label{ample=IT0} Let $\FF$ be a coherent sheaf on an elliptic curve $C$.
\begin{enumerate}
\item If $\FF$ is locally free, then $\FF$ is  $IT^0$ if and only if it is ample.
\item If $\FF$ is  $IT^0$, so is every quotient sheaf of $\FF$.
\end{enumerate}
\end{lem}
\begin{proof}
(1) follows from {\cite[Lemma 4.3]{CH2}}. (2) follows from the fact that $H^i$ vanishes on $C$ for $i>1$.
\end{proof}


\subsection{Reider's theorem}
We recall Reider's theorem and its application on smooth surfaces with $K\equiv 0$. They will be useful in Section 5.
\begin{thm}[{\cite[Theorem 1]{Reider}}]\label{Reider thm}
Let $S$ be a smooth surface and $D$ a nef divisor on $S$.
\begin{enumerate}
\item If $D^2 \geq  5$ and $p$ is a base point of $|K_S + D|$, then there exists an effective divisor $E$ passing through $p$ such that
\begin{align*}
\text{either } &{}D. E = 0, E^2 = - 1,\\
 \text{or } &{}D. E = 1, E^2 = 0.
\end{align*}
\item If $D^2 \geq  10$ and points $p, q$ are not separated by $|K_S + D|$, then there exists an effective divisor $E$ on $S$ passing through $p$ and $q$ such that
\begin{align*}
\text{either } &{}D. E = 0, E^2 = - 2,\\
 \text{or } &{}D. E = 1, E^2 = -1,\\
 \text{or } &{}D. E = 2, E^2 = 0.
\end{align*}
\end{enumerate}
\end{thm}

\begin{cor}\label{Reider K0}
Let $F$ be a smooth surface with $K_F\equiv 0$ and $H$ a nef and big divisor on $F$. Then $|2H|$ is base point free and $|3H|$ gives a birational map.
\end{cor}

\begin{proof}
Note that $K_F\equiv 0$ implies that $H^2$ is an even integer and $H^2\geq 2$. The statement follows from  Reider's theorem directly.
\end{proof}

\section{Linear systems on irregular varieties}
In this section, we recall some results on linear systems on irregular varieties developed by J. A. Chen and Hacon \cite{CH2}. Their results are for  irregular varieties of general type, but can be easily generalized to irregular varieties with $K\equiv 0$. For reader's convenience, we give the proof, but it is essentially the same as in \cite{CH2}. 
\begin{prop}[{cf. \cite[Corollary 2.4]{CH2}}] \label{separate1}Let $X$ be a smooth projective variety with  $K_X\equiv 0$, $a: X \to  A=\Alb(X)$ the Albanese map, and  $L$ a nef and big divisor on $X$. 
\begin{enumerate}
\item Let $F$ be  a fiber of $a$ and $x\in F$. If $x \not \in \Bs|L|_F|$, then $x \not \in \Bs|L+a^*P|$ for general $P\in \Pic^0(A)$;

\item Let $F$ be  a fiber of $a$ and $x\in F$. If $x \not \in \Bs|L+a^*P|$ for all $P\in \Pic^0(A)$, then $a_*(\OO_X(L)\otimes \III_x)$ is $IT^0$;

\item Let $x_1, x_2$ be two points on two different  fibers $F_1, F_2$ of $a$ respectively. If $x_i \not \in \Bs|L|_{F_i}|$ for $i=1,2$, and  $x_1 \not \in \Bs| L + a^*P |$ for all $P\in \Pic^0(A)$, then $|L+a^*P|$ separates $x_1,x_2$ for general $P\in \Pic^0(A)$;

\item Let $x_1, x_2$ be two different points on a  fiber $F$ of $a$. If $x_i \not \in \Bs|L|_{F}|$ for $i=1,2$, $| L|_F |$ separates $x_1, x_2$, and $x_1 \not \in \Bs| L + a^*P |$ for all $P\in \Pic^0(A)$, then $|L+a^*P|$ separates $x_1,x_2$ for  general $P\in \Pic^0(A)$.
\end{enumerate}
\end{prop}

\begin{proof} Set $z=a(F)\in A$. Note that by the proof of Lemma \ref{nef big IT0},
$$a_*\OO_X(L)\otimes \bC(z)\cong H^0(F,\OO_F(L)).$$
We first look at the exact sequence obtained by evaluating at
$x$:
\begin{align*} 0\to \OO_X(L)\otimes \III_x \to \OO_X(L)\to \bC(x)\to 0.\end{align*}
Pushing forward to $A$, we get
\begin{align} 0 \to a_*(\OO_X(L)\otimes \III_x) \to  a_*\OO_X(L) \to \bC(z) \to \cdots\label{exact2}\end{align}
Since $x\not \in \Bs|L|_F |$ and
$$a_*\OO_X(L)\otimes \bC(z)\cong H^0(F,\OO_F(L)),$$
we know that the induced morphism
$$a_*\OO_X(L) \to a_*\OO_X( L) \otimes \bC(z)  \to \bC(z)$$
is non-zero, whence surjective.  By Lemma \ref{nef big IT0}, $a_*\OO_X(L)$ is $IT^0$. Applying Lemma \ref{nonzero} to  $a_*\OO_X(L)$, we get (1).

To see (2), by tensoring (\ref{exact2}) with $P\in \Pic^0(A)$, we have an exact sequence
$$
0 \to a_*(\OO_X(L)\otimes \III_x)\otimes P \to  a_*\OO_X(L)\otimes P \to \bC(z) \to \cdots
$$
By assumption that $x \not \in \Bs|L+a^*P|$, 
$$
H^0(A, a_*(\OO_X(L)\otimes \III_x)\otimes P) \to  H^0(A, a_*\OO_X(L)\otimes P )
$$
is not surjective. Hence 
 $$H^0(A, a_*\OO_X(L)\otimes P )\to H^0( \bC(z)) $$
is non-zero, whence surjective. So we have an exact sequence  
$$
0 \to a_*(\OO_X(L)\otimes \III_x)\otimes P \to  a_*\OO_X(L)\otimes P \to \bC(z) \to 0.
$$
Taking cohomology, we have $H^i(A, a_*(\OO_X(L)\otimes \III_x)\otimes P) =0$ for $i>0$.

To see (4), we   look  at the exact sequence obtained by evaluating at
$x_2$:
$$0 \to a_*(\OO_X(L)\otimes \III_{x_1,x_2})\to a_*(\OO_X(L)\otimes \III_{x_1})\to \bC(z)\to \cdots$$
The last morphism factors as
$$a_*(\OO_X(L)\otimes \III_{x_1})\to a_*(\OO_X(L)\otimes \III_{x_1})\otimes \bC(z)
\to a_*\OO_X (L)\otimes \bC(z) \to \bC(z).$$
It is obtained by evaluating at $x_2$. The assumption that $|L|_F |$ separates $x_1, x_2$ shows that this is surjective. By applying Lemma \ref{nonzero} to  $ a_*(\OO_X(L)\otimes \III_{x_1})$, we are done.

Finally we consider (3). Assume now that $z=a(x_2)$. Again we have an exact sequence:
$$0 \to a_*(\OO_X(L)\otimes \III_{x_1,x_2})\to a_*(\OO_X(L)\otimes \III_{x_1})\to \bC(z)\to \cdots$$
The last morphism factors as
$$a_*(\OO_X(L)\otimes \III_{x_1})\to a_*(\OO_X(L)\otimes \III_{x_1})\otimes \bC(z)
\to a_*\OO_X (L)\otimes \bC(z) \to \bC(z)$$ 
which is obtained by evaluating at $x_2$. Since $a(x_1) \neq a(x_2)$, it follows that $a_*(\OO_X(L)\otimes \III_{x_1})\otimes \bC(z)\cong a_*(\OO_X(L))\otimes \bC(z)$ and hence the above morphism is surjective. Again, by Lemma \ref{nonzero} we are done. 
\end{proof}

\begin{cor}[{cf. \cite[Theorem 2.8]{CH2}}]\label{separate2} Let $X$ be a smooth projective variety with $K_X\equiv 0$, $a : X \to  A=\Alb(X)$ the Albanese map, $L$ a nef and big divisor, $F$ a  fiber of $a$, and $n\geq 2$ a positive integer. Assume that $\dim X-\dim A\leq 3$. Then
\begin{enumerate}
\item $|nL + a^*P |$ separates two general points on two different  fibers for general $P\in \Pic^0(A)$;

\item $|(n+1)L + a^*P|$ separates two general points on two different  fibers for all $P\in \Pic^0(A)$;

\item If $|nL|_F|$ is birational, then $|nL + a^*P|$ separates two general points on $F$ for general $P\in \Pic^0(A)$;

\item If $|nL|_F|$ is birational, then $|(n+1)L +a^*P|$ separates two general points on $F$ for all $P\in \Pic^0(A)$.
\end{enumerate}
\end{cor}
\begin{proof}
By Lemmas \ref{lemma 1} and  \ref{nef big IT0}, $|mL +a^*P|\neq \emptyset$ for all $m\geq 1$ and all $P\in \Pic^0(A)$. For $m_1, m_2\geq 1$,
consider the map
$$|m_1L +a^*P_1|+|m_2L +a^*P_2|\to |(m_1+m_2)L +a^*(P_1 +P_2)|, $$
for $P_1, P_2\in \Pic^0(A)$. If $x \in  F$ is a general point, then $x\not \in \Bs|L|_F|$ and by Proposition \ref{separate1}(1), $x\not \in \Bs|m_iL +a^*P|$ for $i=1,2$ and for general $P\in \Pic^0(A)$. Therefore $x\not \in \Bs|(m_1+m_2)L +a^*P|$ for all $P\in \Pic^0(A)$.

(1) follows from Proposition \ref{separate1}(3) since for two general points $x_1$ and $x_2$ on two different fibers $F_1$ and $F_2$, we have seen that $x_i \not \in \Bs|nL + a^*P|$ for all $P\in \Pic^0(A)$ and so $x_i \not \in \Bs|nL|_{F_i} |$ for $i=1,2$.

(3) follows from Proposition \ref{separate1}(4) since we assumed that $|nL|_F |$ is birational.

(2) and (4) now follow by considering the map
$$|nL +a^*P_1|+|L +a^*P_2| \to  |(n+1)L +a^*(P_1 +P_2)|.$$
Since $|nL +a^*P_1|$ separates $x_1,x_2$ for general $P_1\in \Pic^0(A)$ and $|L + a^*P_2|$ does not vanish along $x_1, x_2$ for general $P_2\in \Pic^0(A)$,  $|(n+1)L +a^*P|$ separates $x_1,x_2$ for all $P\in \Pic^0(A)$. 
\end{proof}
By Corollary \ref{separate2}, we get the main result of this section.
\begin{cor}\label{m4}
Let $X$ be a smooth projective variety  with $K_X\equiv 0$, $a : X \to  A=\Alb(X)$ the Albanese map, $L$ a nef and big divisor on $X$, and $P\in \Pic^0(X)$. Assume that $\dim X-\dim A\leq 2$. Then
\begin{enumerate}
\item $|mL+P|$ gives a birational map for all $m\geq 4$;
\item $|3L+P|$  separates two general points on two different   fibers of $a$;
\item If $|2L|_F|$ is birational on a fiber $F$ of $a$, then $|3L +P|$  separates two general points on $F$.
\end{enumerate}
\end{cor}

\begin{proof}
For a  fiber $F$ of $a$, $F$ is a smooth variety with $K_F\equiv 0$ and $\dim F\leq 2$.  $|mL|_F|$ gives a  birational map for $m\geq 3$ by Theorem \ref{23}(1). Hence the statements follow from Corollary \ref{separate2} directly.
\end{proof}

Similarly, by Theorem \ref{23}(2), we can easily get the following corollary.
\begin{cor}\label{fiber three}
Let $X$ be a smooth projective variety  with $K_X\equiv 0$, $a : X \to  A=\Alb(X)$ the Albanese map, $L$ a nef and big divisor on $X$, and $P\in \Pic^0(X)$. Assume that $\dim X-\dim A=3$. Then
$|mL+P|$ gives a birational map for all $m\geq 6$.
\end{cor}

\begin{remark}
By Corollary \ref{m4}, to prove Theorem \ref{main}, we only need to prove that   $|3L+P|$ separates two general points on a general fiber $F$ of $a$ on which $|2L|_F|$ does not give a birational map, and we only need to consider the cases that $q(X)=1$ or $2$, i.e., the Albanese fiber dimension is $2$ or $1$.
\end{remark}

\section{Irregular varieties of Albanese fiber dimension one}
In this section, we consider irregular varieties with $K\equiv 0$ and of Albanese fiber dimension one. We prove the following theorem.
\begin{thm}\label{fiber one}
Let $X$ be a smooth projective variety  with $K_X\equiv 0$ and $q(X)=\dim X-1$, $L$   a nef and big divisor on $X$.  Then
$|3L+P|$ gives a birational map for all $P\in \Pic^0(X)$.
\end{thm}
\begin{proof}
Every fiber of the Albanese map $a: X\to A=\Alb(X)$ is an elliptic curve. By Corollary \ref{m4}(2), we only need to prove that  $|3L+P|$ separates two general points on a  fiber $F$ of $a$. Also we may assume that $L\cdot F=1$ otherwise  $|2L|_F|$ gives an embedding on $F$ and we are done by Corollary \ref{m4}(3). 

By Lemma \ref{L-F nef} below, $a_*\OO_X(3L)\otimes P$ is generated by global sections for all $P\in \Pic^0(A)$. 
Hence 
\begin{align*}
H^0(X, \OO_X(3L)\otimes a^*P){}&\simeq H^0(A, a_*\OO_X(3L)\otimes P)\\
{}&\to a_*\OO_X(3L)\otimes P\otimes \CC(z)\simeq H^0(F, \OO_F(3L))
\end{align*}
is surjective where $z=a(F)\in A$. Since  $|3L|_F|$ gives an embedding on $F$, $|3L+a^*P|$ separates two points on $F$. We complete the proof.
\end{proof}

\begin{lem}\label{L-F nef}
Let $X$ be a smooth projective variety  with $K_X\equiv 0$, $a : X \to  A=\Alb(X)$ the Albanese map, and $L$ a nef and big divisor on $X$. For a general fiber $F$ of $a$, assume that $h^0(F, \OO_F(L))=1$. Then
$a_*\OO_X(mL)\otimes P$ is generated by global sections for all $P\in \Pic^0(A)$ and all $m\geq 2$.
\end{lem}
\begin{proof}
By the assumption $h^0(F, \OO_F(L))=1$ and the proof of Lemma \ref{nef big IT0}, $a_*\OO_X(L)$ is an $IT^0$ line bundle on $A$. We may write $a_*\OO_X(L)=\OO_A(D)$ where $D$ is a divisor on $A$. By Lemma \ref{lemma 1},
$
h^0(A, \OO_A(D)\otimes P)>0
$
for all $P\in \Pic^0(A)$, and hence $D$ is an ample divisor.

Since
$$
h^0(X, \OO_X(L-a^*D))=h^0(A, a_*\OO_X(L)\otimes\OO_A(-D))=1,
$$
$L-a^*D$ is an effective divisor on $X$. Write $L-a^*D\sim E\geq 0$, $(X, tE)$ is klt for a sufficiently small $t>0$. Assume that $E$ is not nef, by Cone Theorem (cf. \cite{KMM}), there exists a rational curve $C$ such that $E.C<0$. Since $A$ is an abelian variety, $C$ must be contracted by $a$ and then $E\cdot C=L\cdot C\geq 0$, a contradiction. Hence $L-a^*D$ is nef and $mL-a^*D$ is nef and big for $m\geq 2$. By Kawamata--Viehweg vanishing theorem, 
$$
R^ia_*\OO_X(mL-a^*D)=0
$$
for $i>0$ and 
$$
H^i(A, a_*\OO_X(mL)\otimes \OO_A(-D)\otimes P)\simeq H^i(X, \OO_X(mL-a^*D)\otimes a^*P)=0
$$
for $i>0$ and all $P\in \Pic^0(A)$. By \cite[Theorem 2.1]{P},  $a_*\OO_X(mL)\otimes P$ is generated by its global sections for all $P\in \Pic^0(A)$. 
\end{proof}

\section{Irregular threefolds of Albanese fiber dimension two}
In this section, we consider irregular $3$-folds with $K\equiv 0$ and of Albanese fiber dimension two. We prove the following theorem.
\begin{thm}\label{fiber two}
Let $X$ be a smooth $3$-fold  with $K_X\equiv 0$ and $q(X)=1$, $L$ a nef and big divisor on $X$. Then
$|3L+P|$ gives a birational map for all $P\in \Pic^0(X)$.
\end{thm}

Denote by $a: X\to C=\Alb(X)$ the Albanese map where $C$ is an elliptic curve. A fiber $F$ of $a$ is a surface with $K_F\equiv 0$, whence an abelian surface, a bielliptic surface, a K3 surface, or an Enriques surface. Write  $H:=L|_F$. By Corollary \ref{m4}, we need to consider the case when $|2H|$ does not give a birational map on $F$. Recall that $|3H|$ always gives a birational map.
\subsection{Key lemma}The following is the key lemma of this section.
\begin{lem}\label{key lemma}
Keep the notation as above. Fix  $P\in \Pic^0(C)$ and fix a general fiber $F$, assume that
$|2H|$  gives a generically finite morphism $\phi_{|2H|}$ of degree $2$, and every section in 
$$H^0(F, \OO_F(3H))\backslash \Image (m_{12})$$
separates  two points lying in one general fiber of $\phi_{|2H|}$, where $m_{12}$ is the multiplication map of sections
$$
m_{12}:H^0(F, \OO_F(H))\otimes H^0(F, \OO_F(2H))\to H^0(F, \OO_F(3H)).
$$
Then $|3L+a^*P|$ separates  two  general points on $F$.
\end{lem}

Note that sections in  $H^0(F, \OO_F(2H))$ do not separate  two points lying in the same fiber of $\phi_{|2H|}$, hence the assumption of Lemma \ref{key lemma} means that  $\Image (m_{12})$ contains exactly the sections in $H^0(F, \OO_F(3H))$ that do not
separate  two  points lying in one general fiber of $\phi_{|2H|}$.
\begin{proof}
Since $|2H|$ separates  two   points on $F$ not lying in the same fiber of $\phi_{|2H|}$, by Proposition \ref{separate1}(4) and the argument in the proof of Corollary \ref{separate2}, $|3L+a^*P|$ separates  two  points on $F$ not lying in the same fiber of $\phi_{|2H|}$. We only need to find a section in  $H^0(X, \OO_X(3L)\otimes a^*P)$ that separates   two points lying in one general fiber of $\phi_{|2H|}$.

Consider the exact sequence
$$
0\to \FF \to a_*\OO_X(3L)\to \QQ\to 0
$$
where $\FF$ is the image of the multiplication morphism
$$
a_*\OO_X(L)\otimes a_*\OO_X(2L)\to a_*\OO_X(3L).
$$
Denote  $z=a(F)\in C$ a general point. Note that by definition, 
$
\QQ\otimes \CC(z)
$
is the cokernel of 
$$
m_{12}:H^0(F, \OO_F(H))\otimes H^0(F, \OO_F(2H))\to H^0(F, \OO_F(3H)),
$$
which is non-zero since $H^0(F, \OO_F(3H))$ gives a birational map but $\Image(m_{12})$ does not.
Hence $\QQ\neq 0$ and moreover, $\QQ$ is not torsion since $z$ is general. 
Consider the exact sequence
$$
0\to \TT\to \QQ\to  \QQ/ \TT\to 0,
$$
where $\TT$ is the torsion subsheaf of $\QQ$. Then $ \QQ/ \TT$ is a non-zero $IT^0$ vector bundle on $C$ by Lemma \ref{ample=IT0}(2). By Lemma \ref{lemma 1}, $H^0(C, (\QQ/ \TT)\otimes P)\neq 0$. Fix a non-zero section $\sigma_0\in H^0(C, (\QQ/ \TT)\otimes P)$, since 
$$
H^0(C, \QQ\otimes P)\to H^0(C, (\QQ/ \TT)\otimes P)
$$
is surjective, $\sigma_0$ lifts to $\sigma\in H^0(C, \QQ\otimes P)$.
Since $z$ is general and $(\QQ/ \TT)\otimes P$ is locally free, we may assume that $\sigma_0$ is not zero along $z$, hence $\sigma$ is not zero along $z$.

Since $a_*\OO_X(L)$ and $a_*\OO_X(2L)$ are $IT^0$ by Lemma \ref{nef big IT0}, they are ample by Lemma \ref{ample=IT0}(1). So $a_*\OO_X(L)\otimes a_*\OO_X(2L)$ and $\FF$ are ample and whence $IT^0$.  By taking cohomology,  it follows that 
\begin{align*}
H^0(X, \OO_X(3L)\otimes a^*P)\simeq H^0(C, a_*\OO_X(3L)\otimes P)\to H^0(C, \QQ\otimes P)
\end{align*}
is surjective. 
Hence $\sigma$ lifts to $\bar{\sigma}\in H^0(X, \OO_X(3L)\otimes a^*P)$.
Note that,  $\sigma$ is not zero along $z$ by construction, we have
$$
0\neq {\sigma}(z)\in \QQ\otimes P\otimes \CC(z)
$$
Hence $\bar{\sigma}(z)\in H^0(F, \OO_F(3H))\backslash \Image(m_{12})$, and separates  two   points lying in one general fiber of $\phi_{|2H|}$ by assumption. Hence  $|3L+a^*P|$ separates   two   points lying in one general fiber of $\phi_{|2H|}$.
\end{proof}

To apply Lemma \ref{key lemma}, it is important to check the assumption, which is not trivial. 
In the following 4 subsections, we will deal with 4 classes of surfaces with $K\equiv 0$.

\subsection{K3 surfaces}
 In this subsection, we consider K3 surfaces and prove the following proposition.
\begin{prop}\label{separate K3}
Let $F$ be a K3 surface and $H$  a nef and big divisor on $F$. Assume that  $|2H|$ does not give a birational map. Then  $|2H|$ gives a morphism $\phi_{|2H|}$ of degree $2$, and every section in 
$$H^0(F, \OO_F(3H))\backslash \Image (m_{12})$$
separates  two points lying in one general fiber of $\phi_{|2H|}$.
\end{prop}

Before the proof, we need to understand the structure of $F$ in details in the case that $|2H|$ does not give a birational map. 

We recall some basic properties on Hirzebruch surfaces from \cite[Proposition 1.2]{Reid}. For a Hirzebruch surfaces
$\mathbb{F}_n=\mathbb{P}_{\mathbb{P}^1}(\mathcal{O}_{\mathbb{P}^1}\oplus\mathcal{O}_{
\mathbb{P}^1}(n))$, $n\geq  0$, denote $A$ a  fiber of $\mathbb{F}_n\to \bP^1$ and $B$ the $(-n)$-curve as a section of $\mathbb{F}_n\to \bP^1$.
For any $r \geq 0$, the linear system $|B + (n + r)A|$ is  base point free and defines a morphism
$\phi_{n;r} : \bF_n\to \bP^{n+2r+1}$.
Except for the case $n=r=0$, the image $\bF_{n;r}$ of $\phi_{n;r}$ is a surface of degree
$n+2r$. Moreover, $\bF_{n;0}\simeq \bP(1,1,n)$ and $\phi_{n;0} : \bF_n\to \bF_{n;0}$ is the contraction of  the curve $B$; $\phi_{n;r} : \bF_n\to  \bF_{n;r}$ is an isomorphism if $r>0$.

\begin{prop}\label{cases K3}
Let $F$ be a K3 surface and $H$  a nef and big divisor on $F$. Assume that  $|2H|$ does not give a birational map. Then  $|2H|$ gives a  morphism of degree $2$. More precisely,  one of the following holds:
\begin{enumerate}
\item $|H|$ is base point free, $H^2=2$, $|H|$ defines $\phi_{|H|}: F\to \bP^2$, and $|2H|$ defines $\phi_{|2H|}: F\to \bP^2\subset \bP^5$, where the last inclusion is Veronese embedding;

\item $|H|=|2P+\Gamma|$, $H^2=2$, $|2H|$ defines $\phi_{|2H|}: F\to \bP(1,1,4)\subset \bP^5$ where the last inclusion is defined by $\OO(4)$ on $\bP(1,1,4)$;

\item $|H|=|(1+\frac{d}{2})P+\Gamma|$, $H^2=d>2$, $|2H|$ defines $\phi_{|2H|}: F\to \bF_{4}\subset \bP^{1+2d}$  where the last inclusion is defined by $|(2+d)A+B|$ on $\bF_4$.
\end{enumerate}
Here $|P|$ is an elliptic pencil (i.e. a base point free linear system in which the general element is a smooth elliptic curve) and $\Gamma$ is a smooth rational curve with $P\cdot\Gamma=1$.
\end{prop}
\begin{proof}
By Corollary \ref{Reider K0}, $|2H|$ is base point free. By \cite[Proposition 2.6]{SD}, a general member of $|2H|$ is irreducible. By \cite[4.1]{SD} and the assumption that $|2H|$ does not give a birational map, $|2H|$ gives a  morphism of degree $2$ and its image  $\phi_{|2H|}(F)$ has degree $2d$ in $\bP^{1+2d}$ where $d=H^2\geq 2$. In this case $|2H|$ is said to be {\it hyperelliptic}. 
By a theorem of del Pezzo (see \cite[Theorem 1.3]{Reid}), $\phi_{|2H|}(F)$ is either $\bP^2$, $\bP^2$ in its Veronese embedding, or one of the $\bF_{n;r}$. 

The first case does not happen because $1+2d> 2$.
For the second case, we have $\OO_F(2H)=\phi^*_{|2H|}(\OO_{\bP^2}(2))$, and $\OO_F(H)=\phi^*_{|2H|}(\OO_{\bP^2}(1))$ since there is no non-trivial torsion divisor on $F$. This is exactly statement (1) of the proposition.

Now we assume that $\phi_{|2H|}(F)=\bF_{n;r}\subset \bP^{1+2d}$ for some $n$ and  $r$, then $n+2r=2d$ by definition and $n \leq 4$ by \cite[Corollary 2.4]{Reid}.  Moreover, $\phi_{|2H|}$ factors through $\phi_{n;r}$ by a morphism $\psi: F\to \bF_{n}$, which is a double cover. We have $|2H|=|\psi^*((n+r)A+B)|$ and $\psi_*\OO_F=\OO_{\bF_{n}}\oplus \OO_{\bF_{n}}(-(n+2)A-2B)$ (cf. \cite[Section 2]{Reid}).

By \cite[Proposition 0.1]{Nikulin} or \cite[3.8 Theorem (d)]{ReidChapter} and its proof, either $|H|$ has no fixed part  or $|H|=|(1+\frac{d}{2})P+\Gamma|$ where $|P|$ is an elliptic pencil and $\Gamma$ is a smooth rational curve with $P\cdot \Gamma=1$.

If  $|H|$ has no fixed part,   then by \cite[Proposition 2.6, Corollary 3.2]{SD}, $|H|$ is base point free and a general member of $|H|$ is irreducible. Then \cite[3.9.6]{SD} implies that $H\cdot E>1$ for every irreducible curve $E$ satisfying $E^2=0$. On the other hand, by \cite[Proposition 5.6]{SD}, the second case in  \cite[Theorem 5.2]{SD} for the linear system $|2H|$ does not happen. Hence there exists an irreducible curve $E$ such that $E^2=0$ and $2H\cdot E=2$, which is a contradiction.

Hence $|H|=|(1+\frac{d}{2})P+\Gamma|$ where $|P|$ is an elliptic pencil and $\Gamma$ is a smooth rational curve with $P\cdot\Gamma=1$. Note that $|\psi^*A|$ is an elliptic pencil on $F$ since $|\psi^*A|$ is base point free and $h^0(F,\OO_F(\psi^*A))=2$. If $|\psi^*A|\neq |P|$, then $\psi^*A\cdot P\geq 1$ and $2H\cdot\psi^*A=((2+d)P+2\Gamma)\cdot\psi^*A\geq 2+d$. On the other hand, $2H\cdot\psi^*A=\psi^*B\cdot\psi^*A= 2$, a contradiction. Hence $|\psi^*A|= |P|$. We have 
$$(2+d)\psi^*A+2\Gamma\sim 2H\sim \psi^*((n+r)A+B).$$
Hence $h^0(F, \OO_F(\psi^*((n+r-2-d)A+B)))>0$. Note that by projection formula, we have 
\begin{align*}
{}&h^0(F, \OO_F(\psi^*((n+r-2-d)A+B)))\\
={}&h^0({\bF_{n}}, \OO_{\bF_{n}}((n+r-2-d)A+B))+h^0({\bF_{n}}, \OO_{\bF_{n}}((r-4-d)A-B))\\
={}&h^0({\bF_{n}}, \OO_{\bF_{n}}((n+r-2-d)A+B)).
\end{align*}
This implies that $n+r-2-d\geq 0$. Combining with $n+2r=2d$, this implies that $n\geq 4$. Hence $n=4$ and $r=d-2$. Then we get the statements (2) and (3) of the proposition.
\end{proof}

\begin{proof}[Proof of Proposition \ref{separate K3}]We discuss case by case.

In Case (1) of Proposition \ref{cases K3}, we know that $H^0(F, \OO_F(H))$ is generated by $\{x, y, z\}$ and $H^0(F, \OO_F(2H))$ is generated by $\{x^2, xy, xz, y^2, yz, z^2 \}$. Then $\Image(m_{12})\subset H^0(F, \OO_F(3H))$ is generated by 
$$\{x^i y^j z^k\mid i, j, k\geq 0, i+j+k=3\}$$ which is $10$-dimensional. Since $h^0(F, \OO_F(3H))=11$, there exists a section $\sigma$ such that $H^0(F, \OO_F(3H))$ is generated by $\Image(m_{12})$ and $\sigma$. Since $|3H|$ gives a birational map on $F$, $\sigma$ separates two  points lying in one general fiber of $\phi_{|2H|}$,  and so does every section in 
$H^0(F, \OO_F(3H))\backslash \Image (m_{12})$.

In Case (2) of Proposition \ref{cases K3}, since $|P|$ is an elliptic pencil, $H^0(F, \OO_F(P))$ is generated by $\{x, y\}$. Since $h^0(F, \OO_F(\Gamma))=1$,  $H^0(F, \OO_F(\Gamma))$ is generated by $\gamma$. Hence $H^0(F, \OO_F(H))$ contains $\{x^2\gamma, xy\gamma, y^2\gamma\}$ and it turns out to be a basis by dimension computation.  Moreover, $H^0(F, \OO_F(2H))$  contains $\{x^i y^j \gamma^2 \mid i, j\geq 0, i+j=4\}$ which is 5-dimensional. We can choose a section $z$ to complete a basis of  $H^0(F, \OO_F(2H))$. Then $\Image(m_{12})\subset H^0(F, \OO_F(3H))$ is generated by  $$\{x^i y^j \gamma^3, x^k y^l \gamma z \mid i, j,k,l\geq 0, i+j=6, k+l=2\}$$ which is 10-dimensional. Since $h^0(F, \OO_F(3H))=11$, there exists a section $\sigma$ such that $H^0(F, \OO_F(3H))$ is generated by $\Image(m_{12})$ and $\sigma$. Since $|3H|$ gives a birational map on $F$, $\sigma$ separates two  points lying in one general fiber of $\phi_{|2H|}$,  and so does every section in 
$H^0(F, \OO_F(3H))\backslash \Image (m_{12})$.

In Case (3) of Proposition \ref{cases K3}, set $|H'|=|2P+\Gamma|$, then $|H'|$ satisfies Case (2) of Proposition \ref{cases K3}. We have seen that $H^0(F, \OO_F(H'))$ is generated by $\{x^2\gamma, xy\gamma, y^2\gamma\}$, $H^0(F, \OO_F(2H'))$ is generated by $$\{x^i y^j \gamma^2, z \mid i, j\geq 0, i+j=4\},$$ and $H^0(F, \OO_F(3H))$ is generated by  $$\{x^i y^j \gamma^3, x^k y^l \gamma z, \sigma \mid i, j,k,l\geq 0, i+j=6, k+l=2\}.$$ Here $\sigma$ separates two  points lying in one general fiber of $\phi_{|2H'|}$. Note that $\phi_{|2H'|}$ is the composition of $\phi_{|2H|}$ with the contraction $\bF_4\to \bP(1,1,4)$ which contracts  the curve $B$. Hence $\sigma$ separates two  points lying in one general fiber of $\phi_{|2H|}$. Since $|3H|=|3H'+(\frac{3d}{2}-3)P|$ and  $H^0(F, \OO_F(P))$ is generated by $\{x, y\}$, $H^0(F, \OO_F(3H))$ contains 
\begin{align*}
\Big\{x^i y^j \gamma^3, x^k y^l \gamma z, x^m y^n \sigma\text{ } \Big |\text{ } i,  {}& j,k,l,m,n\geq 0, i+j=\frac{3d}{2}+3, \\
{}&k+l=\frac{3d}{2}-1, m+n=\frac{3d}{2}-3\Big\}.
\end{align*} 
This turns out to be a basis by dimension computation. By similar argument,  $\Image(m_{12})$ is generated by  
$$\Big\{x^i y^j \gamma^3, x^k y^l  \gamma z \text{ }\Big| \text{ } i, j,k,l\geq 0, i+j=\frac{3d}{2}+3, k+l=\frac{3d}{2}-1\Big\}.$$ 
In other words,  $H^0(F, \OO_F(3H))$ is generated by $\Image(m_{12})$ and $$\Big \{x^m y^n \sigma \text{ } \Big |\text{ }  m,n \geq 0,m+n=\frac{3d}{2}-3\Big \}.$$ Since $\sigma$ separates  two points lying in one general fiber of $\phi_{|2H|}$,  so does every section in 
$H^0(F, \OO_F(3H))\backslash \Image (m_{12})$.
\end{proof}

\subsection{Enriques surfaces}
 In this subsection, we consider Enriques surfaces and prove the following proposition.

\begin{prop}\label{separate Enriques}
Let $F$ be an Enriques surface and $H$   a nef and big divisor on $F$. Assume that  $|2H|$ does not give a birational map. Then  $|2H|$ gives a  morphism $\phi_{|2H|}$ of degree $2$, and every section in 
$$H^0(F, \OO_F(3H))\backslash \Image (m_{12})$$
separates  two points lying in one general fiber of $\phi_{|2H|}$.
\end{prop}
\begin{proof}
By Corollary \ref{Reider K0}, $|2H|$ is base point free. Hence a general member of $|2H|$ is irreducible by \cite[Proposition 1.5.2]{Cos}. Note that $p_a(2H)=1+2H^2\geq 5$. By \cite[Lemma 3.3.3]{Cos} and  base-point freeness of $|2H|$, $\phi_{|2H|}$ is a  morphism of degree 2. By base-point freeness of $|2H|$ again, $|2H|$ is not hyperelliptic (for definition, see \cite[Theorem 4.1]{Cos}). By \cite[Proposition 5.2.1]{Cos}, $p_a(2H)= 5$. By  \cite[Lemma 5.2.7]{Cos}, there exists an irreducible pencil of genus two (with two base point) $|M|$ such that $|2H|=|2M|$. This implies that $|H|=|M|$ or $|H|=|M+K_F|$ since the only non-trivial torsion divisor on $F$ is $K_F$. 
By \cite[Theorem 6.1]{Cos}, there are 2 types for $|M|$: non-special type and special type
(See \cite[Proposition 1.5.4, Definition 1.5.5]{Cos} for definition and property).

Recall that by an elliptic pencil $|P|$ on $F$, we always mean a base point free linear system in which the general element is an elliptic curve. For an elliptic pencil $|P|$ on $F$, there exist effective divisors $E$ and $E'$ such that  $|P|=|2E|=|2E'|$, $E' \in |E+K_F|$, and $2E, 2E'$ are the only multiple fibers of $|P|$ (cf. \cite[Proposition 1.6.3]{Cos}).

Firstly we consider the case that $|M|$ is non-special. Then by \cite[6.3.1]{Cos}, there exist two elliptic pencils $|2E_1|$ and $|2E_2|$ such that $|M|=|E_1+E_2|$. In this case $|M+K_F|=|E_1+E_2+K_F|=|E_1+E'_2|$ is of the same form as $|M|$. Hence we may assume that $|H|=|M|$. Since 
$$\dim H^0(F, \OO_F(E_i))=\dim H^0(F, \OO_F(E'_i))=1,$$
we write their generators  by $e_i$ and $e_i'$ respectively for $i=1,2$. Since $|2H|=|2E_1+2E_2|=|2E'_1+2E'_2|=|E_1+E_2+E'_1+E'_2|$, $H^0(F, \OO_F(2H))$ contains 
$$\{e^2_1e^2_2, e^{\prime 2}_1e^{\prime 2}_2, e^{2}_1e^{\prime 2}_2, e^{\prime 2}_1e^{2}_2, e_1e_2e'_1e'_2\}$$
and it turns out to be a basis by dimension computation.
For simplicity, we may write $x_1=e^2_1e^2_2$, $x_2=e^{\prime 2}_1e^{\prime 2}_2$, $x_3=e^{2}_1e^{\prime 2}_2$, $x_4=e^{\prime 2}_1e^{2}_2$, and $x_0=e_1e_2e'_1e'_2$. Then they satisfy equations   $x_0^2=x_1x_2$ and $x_0^2=x_3x_4$ which actually define $\phi_{|2H|}(F)\subset \bP^4$ (cf. \cite[1.4]{D}). For simplicity, we formally denote $\sqrt{x_1}=e_1e_2$ and define $\sqrt{x_i}$ similarly for $i=2,3,4$. 
Since $|H|=|E_1+E_2|=|E'_1+E'_2|$, $H^0(F, \OO_F(H))$ contains $\{\sqrt{x_1}, \sqrt{x_2}\}$ and this turns out to be a basis by dimension computation. Hence $\Image(m_{12})\subset H^0(F, \OO_F(3H))$ is generated by 
$$\{x_i\sqrt{x_1}, x_j\sqrt{x_2}\mid 0\leq i\leq 4, 2\leq j\leq 4\}$$
 which is $8$-dimensional (note that $x_0\sqrt{x_2}=x_2\sqrt{x_1}$ and $x_1\sqrt{x_2}=x_0\sqrt{x_1}$). On the other hand, recall that $|2H+K_F|$ gives a birational map on $F$ (cf. \cite[Proof of Lemma 5.2.7]{Cos}), there is a section $\sigma\in H^0(F, \OO_F(2H+K_F))$ which separates  two points lying in one general fiber of $\phi_{|2H|}$. Also by $|H+K_F|=|E_1+E'_2|=|E'_1+E_2|$,
 it follows that $H^0(F, \OO_F(H+K_F))$ is generated by $\{\sqrt{x_3}, \sqrt{x_4}\}$.
 Since $|3H|=|(2H+K_F)+(H+K_F)|$, $H^0(F, \OO_F(3H))$ contains $\{\sigma\sqrt{x_3}, \sigma\sqrt{x_4}\}$. Hence $H^0(F, \OO_F(3H))$ is generated by $\Image(m_{12})$ and  $\{\sigma\sqrt{x_3}, \sigma\sqrt{x_4}\}$ by dimension computation. Since $\sigma$ separates  two points lying in one general fiber of $\phi_{|2H|}$, so does every section in 
$H^0(F, \OO_F(3H))\backslash \Image (m_{12})$.

Finally we  consider the case that $|M|$ is special. By \cite[6.4.1]{Cos}, there exists an elliptic pencil $|P|=|2E|$ such that $|M|=|P+\Theta+K_F|$ where $\Theta$ is a smooth rational curve with $P.\Theta=2$. 
Since 
$$\dim H^0(F, \OO_F(E))=\dim H^0(F, \OO_F(E'))=\dim H^0(F, \OO_F(\Theta))=1,$$
we write their generators  by $e$, $e'$, and $\theta$ respectively. Then $H^0(F, \OO_F(P))$ is generated by $\{e^2, e^{\prime 2}\}$ and  $H^0(F, \OO_F(M+K_F))$ is generated by $\{e^2\theta,  e^{\prime 2}\theta\}$ since $|M+K_F|=|P+\Theta|$.  Since $|P+K_F|=|E+E'|$, $H^0(F, \OO_F(P+K_F))$ contains $ee'$ and hence  $H^0(F, \OO_F(M))$ contains $ee'\theta$. We may choose a section $\eta$ to complete a basis of  $H^0(F, \OO_F(M))$. Since $|2H|=|2M|=|2(M+K_F)|$, $H^0(F, \OO_F(2H))$ contains $\{e^4\theta^2, e^2 e^{\prime 2}\theta^2,  e^{\prime 4}\theta^2\}$ and $\{e^2 e^{\prime 2}\theta^2, ee'\theta\eta, \eta^2\}$ which sum up to a basis.
For simplicity, we write $x_0=e^2 e^{\prime 2}\theta^2$, $x_1=e^4\theta^2$, $x_2= e^{\prime 4}\theta^2$, $x_3=ee'\theta\eta$, $x_4=\eta^2$. Then they satisfy equations   $x_0^2=x_1x_2$ and $x_3^2=x_0x_4$   which actually define $\phi_{|2H|}(F)\subset \bP^4$ (cf. \cite[1.4]{D}).
We can formally define $\sqrt{x_0}=ee'\theta$ and define $\sqrt{x_i}$ for $i=1,2,4$ similarly. Rewrite what we have known, $H^0(F, \OO_F(M+K_F))$ is generated by $\{\sqrt{x_1}, \sqrt{x_2}\}$,  $H^0(F, \OO_F(M))$ is generated by $\{\sqrt{x_0}, \sqrt{x_4}\}$. On the other hand, recall that $|2H+K_F|$ gives a birational map on $F$ (cf. \cite[Proof of Lemma 5.2.7]{Cos}), there is a section $\sigma\in H^0(F, \OO_F(2H+K_F))$ which separates  two points lying in one general fiber of $\phi_{|2H|}$. 

If $|H|=|M+K_F|$, then $\Image(m_{12})\subset H^0(F, \OO_F(3H))$ is generated by $$\{x_i\sqrt{x_1}, x_j\sqrt{x_2}\mid 0\leq i\leq 4, 2\leq j\leq 4\}$$ which is $8$-dimensional. Note that $H^0(F, \OO_F(3H))$  contains $\{\sigma\sqrt{x_0}, \sigma\sqrt{x_4}\}$ since $|3H|=|(2H+K_F)+M|$.  Hence $H^0(F, \OO_F(3H))$ is generated by $\Image(m_{12})$ and  $\{\sigma\sqrt{x_0}, \sigma\sqrt{x_4}\}$ by dimension computation. 

If $|H|=|M|$, then $\Image(m_{12})\subset H^0(F, \OO_F(3H))$ is generated by $$\{x_i\sqrt{x_0}, x_j\sqrt{x_4}\mid 0\leq i\leq 4, j=1,2, 4\}$$ which is $8$-dimensional. Note that $H^0(F, \OO_F(3H))$  contains $\{\sigma\sqrt{x_1}, \sigma\sqrt{x_2}\}$ since $|3H|=|(2H+K_F)+(M+K_F)|$. Hence $H^0(F, \OO_F(3H))$ is generated by $\Image(m_{12})$ and  $\{\sigma\sqrt{x_1}, \sigma\sqrt{x_2}\}$ by dimension computation. 

Since $\sigma$ separates  two points lying in one general fiber of $\phi_{|2H|}$ in both cases, so does every section in 
$H^0(F, \OO_F(3H))\backslash \Image (m_{12})$.
\end{proof} 
\begin{remark}Lemma \ref{key lemma} and Proposition \ref{separate Enriques} imply that, if $F$ is an Enriques surface, then $|3L+a^*P|$ separates two general points on $F$ for all $P\in \Pic^0(C)$. In fact, there is a simpler proof: if $|2H|$ gives a birational map, then $|3L+a^*P|$ separates two general points on $F$ for all $P\in \Pic^0(C)$ by Corollary \ref{m4}(4); if $|2H|$ does not give a birational map, then $|2H+K_F|$  gives a birational map (cf. \cite[Proof of Lemma 5.2.7]{Cos}), hence   $|2L+K_X+a^*P|$ separates two general points on $F$ for general $P\in \Pic^0(C)$ by Corollary \ref{separate2}(4) and $|3L+a^*P|$ separates two general points on $F$ for all $P\in \Pic^0(C)$ by considering the map
$$
|2L+K_X+a^*P_1|+|L-K_X+a^*P_2|\to|3L+a^*(P_1+P_2)|
$$
for $P_1, P_2\in \Pic^0(C)$.
\end{remark}

\subsection{Abelian surfaces}
 In this subsection, we consider abelian surfaces and prove the following proposition.

\begin{prop}\label{separate ab}
Let $F$ be an abelian surface and $H$ a nef and big divisor on $F$ with $H^2\geq 4$. Assume that  $|2H|$ does not give a birational map. Then  $|2H|$ gives a finite morphism $\phi_{|2H|}$ of degree $2$, and every section in 
$$H^0(F, \OO_F(3H))\backslash \Image (m_{12})$$
separates  two  points lying in one general fiber of $\phi_{|2H|}$.
\end{prop}
We remark that for a divisor on an abelian surface or a bielliptic surface, ampleness is equivalent to nef-and-bigness. The structure of $F$ on which $|2H|$ does not give a birational map is very clear by the following theorem.
\begin{thm}[{\cite[Theorem 1]{Oh}}]\label{cases ab}
Let $F$ be an abelian surface and $H$ an ample divisor on $F$. Assume that  $|2H|$ does not give a birational map. Then either $H^2=2$, or $F\simeq C_1\times C_2$ and $H=D_1\times C_2+C_1\times D_2$, where $C_i$ is an elliptic curve and $D_i$ is an ample divisor on $C_i$ for $i=1,2$, and $\deg D_1=1$.
\end{thm}

\begin{proof}[Proof of Proposition \ref{separate ab}]By Theorem \ref{cases ab}, $F\simeq C_1\times C_2$ and $H=D_1\times C_2+C_1\times D_2$, where $C_i$ is an elliptic curve and $D_i$ is an ample divisor on $C_i$ for $i=1,2$, and $\deg D_1=1$, $\deg D_2=\frac{H^2}{2}\geq 2$. Then $|2H|$ defines a finite morphism $C_1\times C_2\to \bP^1\times C_2$ of degree two. Note that two points lie in one fiber of $\phi_{|2H|}$ if and only if they lie in one fiber of $C_1\to \bP^1$ defined by $|2D_1|=|2H|_{C_1}|$, here $C_1$ is viewed as a fiber of $C_1\times C_2\to C_2$. 

Now  $H^0(C_1, \OO_{C_1}(D_1))$ is generated by a section $x$, $H^0(C_1, \OO_{C_1}(2D_1))$ is generated by $\{x^2, y\}$, and $H^0(C_1, \OO_{C_1}(3D_1))$ is generated by $\{x^3, xy, \sigma\}$, where $\sigma$ separates two points lying in the same fiber of $C_1\to \bP^1$ defined by $|2D_1|$. Also the multiplication map of sections
$$
H^0(C_2, \OO_{C_2}(D_2))\otimes H^0(C_2, \OO_{C_2}(2D_2))\to H^0(C_2, \OO_{C_2}(3D_2))
$$
is surjective. Since 
$$
H^0(F, \OO_F(iH))\simeq H^0(C_1, \OO_{C_1}(iD_1))\otimes H^0(C_2, \OO_{C_2}(iD_2))
$$
for $i=1,2,3$, it is easy to see that $\Image(m_{12})\subset H^0(F, \OO_F(3H))$ is generated by $\{x^3, xy\}\otimes H^0(C_2, \OO_{C_2}(3D_2))$ and $ H^0(F, \OO_F(3H))$ is generated by $\Image(m_{12})$ and  $\sigma \otimes H^0(C_2, \OO_{C_2}(3D_2))$. 
Since $\sigma$ separates  two points lying in the same fiber of $C_1\to \bP^1$,  
$H^0(F, \OO_F(3H))\backslash \Image (m_{12})$ separates  two points lying in  one general fiber of $\phi_{|2H|}$.
\end{proof}

\subsection{Bielliptic surfaces}
In this subsection, we consider bielliptic surfaces. In fact, we could not apply Lemma \ref{key lemma} directly in this case. The key point is that we should replace $m_{12}$ in Lemma \ref{key lemma} by another multiplication map $m'_{12}$, see Proposition \ref{separate biell} and proof of Proposition \ref{key lemma biell}. We can still prove the following proposition. 

\begin{prop}\label{key lemma biell}
Keep the notation as the beginning of Section 5. Fix  $P\in \Pic^0(C)$ and fix a general fiber $F$, assume that $F$ is a bielliptic surface and $|2H|$ does not give a birational map.
Then $|3L+a^*P|$ separates  two  general points on $F$.
\end{prop}
We recall the classification of bielliptic surfaces.
\begin{thm}[{see \cite{Beauville}}] Given a bielliptic surface $F$, there exist two elliptic curves $A$, $B$, and an abelian group $G$ acting on $A$ and on $B$ such that:
\begin{enumerate}
\item $A/G$ is elliptic and $B/G \simeq \bP^1$;
\item $S\simeq (A \times B)/G$, where $G$ acts on $A \times B$ componentwisely.
\end{enumerate}
\end{thm}

Denote by $\Phi: S\to (A/G)$, $\Psi: S\to (B/G)$ the two natural projections. Since $A\to (A/G)$ is \'etale, all fibers of $\Phi$ are smooth. All smooth fibers of $\Phi$ (respectively of $\Psi$) are isomorphic to $B$ (resp., to $A$). We will  denote by $A$ or $B$ the class in $\Num(S)$ (the group of numerical classes of divisors on $S$) of a fiber of $\Psi$ or $\Phi$ respectively.
We list all the possibilities of $G$ and basis of $\Num(S)$ in the following table (see \cite{Serrano}).
\smallskip
{
\begin{center}
\begin{tabular}{lll}
\hline
Type & $G$ & Basis of $\Num(S)$\\
\hline
1 & $\ZZ_2$ & $\{(1/2)A, B\}$\\

2 & $\ZZ_2\times\ZZ_2$ & $\{(1/2)A, (1/2)B\}$\\

3 & $\ZZ_4$ & $\{(1/4)A, B\}$\\

4 & $\ZZ_4\times \ZZ_2$ & $\{(1/4)A, (1/2)B\}$\\

5 & $\ZZ_3$ & $\{(1/3)A, B\}$\\

6 & $\ZZ_3\times \ZZ_3$ & $\{(1/3)A, (1/3)B\}$\\

7 & $\ZZ_6$ & $\{(1/6)A, B\}$\\
\hline
\end{tabular}
\end{center}
\centerline{Table 1}}

First we prove the following proposition.
\begin{prop}\label{separate biell}
Let $F$ be a bielliptic surface and $H$ a nef and big divisor on $F$. Assume that  $|2H|$ does not give a birational map. Then  $|2H|$ gives a morphism $\phi_{|2H|}$ of degree $2$, and every section in 
$$H^0(F, \OO_F(3H))\backslash \Image (m'_{12})$$
separates  two  points lying in one general fiber of $\phi_{|2H|}$, where 
$m'_{12}$ is the multiplication map
$$
m'_{12}:\bigoplus_{i=0}^{r-1}H^0(F, \OO_{F}(H+iK_F))\otimes H^0(F,\OO_{F}(2H-iK_F))\to H^0(F, \OO_{F}(3H)).$$ 
Here $r$ is the global index of $K_F$, i.e, $r=\min\{m\in \ZZ_{>0}\mid mK_F\sim 0\}$.
\end{prop}

\begin{proof}
By classification, we may assume that $F\simeq (A\times B)/G$ and  $H\equiv \frac{m}{r}A+\frac{n}{s}B$, where $m,n$ are positive integers and $\{\frac{1}{r}A,\frac{1}{s}B\}$ is the basis of $\Num(F)$. In this case $A\cdot B=|G|=rs$ with $r\geq 2$ and $r$ is also the global index of $K_F$. Since $|2H|$ is not birational, by Reider's theorem (Theorem \ref{Reider thm}), there exists a 
base point free pencil $E$ such that $2H\cdot  E=2$ (cf. \cite[Corollary 2]{Reider}). Since $E^2=0$, either $E\equiv A$ or  $E\equiv B$. If $E\equiv A$, then 
$$1=E\cdot H=A\cdot ( \frac{m}{r}A+\frac{n}{s}B)=rn\geq 2,$$ which is a contradiction. Hence $E\equiv B$ and 
$$1=E\cdot H=B\cdot ( \frac{m}{r}A+\frac{n}{s}B)=sm.$$
This implies that $m=s=1$ which means that we are now in Type $1,3,5,7$ of Table 1.
Hence $H\equiv \frac{1}{r}A+nB$ with $n=\frac{H^2}{2}\geq 1$. 

For $0\leq i \leq r-1$, we consider two points $x,y$ lying in one general fiber of $\phi_{|2H+iK_F|}$, by the above argument using Reider's theorem, $x,y$ must lie in a fiber $B_0$ of $\Phi: S\to (A/G)$. Since $2H-B\equiv \frac{2}{r}A+(2n-1)B$ is ample,
$$
H^0(F, \OO_F(2H+iK_F))\to H^0(B_0, \OO_{B_0}(2H+iK_F))
$$
is surjective and $x,y$ must lie in a fiber of $B_0\to \bP^1$ defined by $|(2H+iK_F)|_{B_0}|=|2H|_{B_0}|$.
Hence $\phi_{|2H+iK_F|}$ is a generically finite morphism of degree 2, and a general fiber of $\phi_{|2H+iK_F|}$ is exactly a fiber of $B_0\to \bP^1$ defined by $|2H|_{B_0}|$ where $B_0$ is some fiber of $\Phi: S\to (A/G)$. In particular, this general fiber is independent of $i$. 

Since $H\cdot  B=1$ and $h^0(F,\OO_F(H))>0$, we may assume that $H$ is effective and its $\Phi$-horizontal part $H^h$ is isomorphic to $A/G$. We may write $H^h\equiv  \frac{1}{r}A+n'B$ with $n'\leq n$. Since $g(H^h)=g(A/G)=0$, $n'=0$ and $H^h\equiv  \frac{1}{r}A$. We have $H=H^h +\Phi^*\OO_{A/G}(D)$ where $D$ is an effective divisor on $A/G$ with $\deg D=n$. Denote $\bar{F}=A\times B$, note that we have the following commutative diagram:
$$\xymatrix{
B\ar[d]&  \bar{F} \ar[l]_q \ar[d]^\tau \ar[r]^p & A    \ar[d]^\pi \\
 B/G& F   \ar[l]_\Psi\ar[r]^\Phi  & A/G.
  }
$$
It is easy to see that $\tau^*\OO_F(H)\simeq \OO_{\bar{F}}(\bar{H})$ where $\bar{H}\sim q^*\OO_B(D_1)+p^*\OO_A(D_2)$ where $D_1$ and $D_2$ are divisors on $B$ and $A$ respectively with $\deg D_1=1$ and $\deg D_2=rn\geq 2$.
We have a multiplication map
$$
\bar{m}_{12}:H^0(\bar{F}, \OO_{\bar{F}}(\bar{H}))\otimes H^0(\bar{F},  \OO_{\bar{F}}(2\bar{H}))\to H^0(\bar{F}, \OO_{\bar{F}}(3\bar{H})).
$$
By the proof of  Proposition \ref{separate ab}, every section in $H^0(\bar{F}, \OO_{\bar{F}}(3\bar{H}))\backslash \Image({\bar{m}_{12}})$ separates two points lying in the same fiber of $B\to \bP^1$ which is defined by $|2\bar{H}|_B|$, here $B$ is viewed as a general fiber of $p$. Note that by projection formula, for $j\geq 0$,
$$
H^0(\bar{F},\OO_{\bar{F}}(j\bar{H}))\simeq \bigoplus_{i=0}^{r-1}H^0(F, \OO_F(jH+iK_F)).
$$
Hence $\Image({\bar{m}_{12}})\cap H^0(F, \OO_F(3H)) =\Image (m'_{12})$. Hence every section in $H^0(F, \OO_{F}(3H))\backslash \Image({m'_{12}})$ separates two points lying in the same fiber of $B\to \bP^1$ which is defined by $|2H|_B|$, here $B$ is viewed as a general fiber of $\Phi$. 
Hence 
every section in 
$$H^0(F, \OO_F(3H))\backslash \Image (m'_{12})$$
separates  two  points lying in one general fiber of $\phi_{|2H|}$.
\end{proof}

The proof of Proposition \ref{key lemma biell} is almost the same with the proof of Lemma \ref{key lemma}.
\begin{proof}[Proof of Proposition \ref{key lemma biell}]
Since $|2H|$ separates  two   points on $F$ not lying in the same fiber of $\phi_{|2H|}$, by Proposition \ref{separate1}(4) and the argument in the proof of Corollary \ref{separate2}, $|3L+a^*P|$ separates  two  points on $F$ not lying in the same fiber of $\phi_{|2H|}$. We only need to find a section in  $H^0(X, \OO_X(3L)\otimes a^*P)$ that separates   two points lying in one general fiber of $\phi_{|2H|}$.

Consider the exact sequence
$$
0\to \FF' \to a_*\OO_X(3L)\to \QQ'\to 0
$$
where $\FF'$ is the image of the multiplication morphism
$$
\bigoplus_{i=0}^{r-1}a_*\OO_X(L+iK_X)\otimes a_*\OO_X(2L-iK_X)\to a_*\OO_X(3L).
$$
Denote  $z=a(F)\in C$ a general point. Note that by definition and adjunction formula, 
$
\QQ'\otimes \CC(z)
$
is the cokernel of 
$$
m'_{12}:\bigoplus_{i=0}^{r-1}H^0(F, \OO_{F}(H+iK_F))\otimes H^0(F,\OO_{F}(2H-iK_F))\to H^0(F, \OO_{F}(3H)),$$
which is non-zero since $H^0(F, \OO_F(3H))$ gives a birational map but $\Image(m'_{12})$ can not separate  two   points on $F$ lying in one general fiber of $\phi_{|2H|}$ by the proof of Proposition \ref{separate biell}.
Hence $\QQ'\neq 0$ and moreover, $\QQ'$ is not torsion since $z$ is general. 
Consider the exact sequence
$$
0\to \TT'\to \QQ'\to  \QQ'/ \TT'\to 0,
$$
where $\TT'$ is the torsion subsheaf of $\QQ'$. Then $ \QQ'/ \TT'$ is a non-zero $IT^0$ vector bundle on $C$ by Lemma \ref{ample=IT0}(2). By Lemma \ref{lemma 1}, $H^0(C, (\QQ'/ \TT')\otimes P)\neq 0$. Fix a non-zero section $\sigma_0\in H^0(C, (\QQ'/ \TT')\otimes P)$, since 
$$
H^0(C, \QQ'\otimes P)\to H^0(C, (\QQ'/ \TT')\otimes P)
$$
is surjective, $\sigma_0$ lifts to $\sigma\in H^0(C, \QQ'\otimes P)$.
Since $z$ is general and $(\QQ'/ \TT')\otimes P$ is locally free, we may assume that $\sigma_0$ is not zero along $z$, hence $\sigma$ is not zero along $z$.

Since $a_*\OO_X(L+iK_X)$ and $a_*\OO_X(2L-iK_X)$ are $IT^0$ by Lemma \ref{nef big IT0}, they are ample by Lemma \ref{ample=IT0}(1). So $\bigoplus_{i=0}^{r-1}a_*\OO_X(L+iK_X)\otimes a_*\OO_X(2L-iK_X)$ and $\FF'$ are ample and whence $IT^0$.  By taking cohomology,  it follows that 
\begin{align*}
H^0(X, \OO_X(3L)\otimes a^*P)\simeq H^0(C, a_*\OO_X(3L)\otimes P)\to H^0(C, \QQ'\otimes P)
\end{align*}
is surjective. 
Hence $\sigma$ lifts to $\bar{\sigma}\in H^0(X, \OO_X(3L)\otimes a^*P)$.
Note that,  $\sigma$ is not zero along $z$ by construction, we have
$$
0\neq {\sigma}(z)\in \QQ'\otimes P\otimes \CC(z)\simeq H^0(F, \QQ'|_F).
$$
Hence $\bar{\sigma}(z)\in H^0(F, \OO_F(3H))\backslash \Image(m'_{12})$, and separates  two   points lying in one general fiber of $\phi_{|2H|}$ by Proposition \ref{separate biell}. Hence  $|3L+a^*P|$ separates   two   points lying in one general fiber of $\phi_{|2H|}$.
\end{proof}

\subsection{Proof of Theorems \ref{fiber two} and \ref{main}}
\begin{proof}[Proof of Theorem \ref{fiber two}]
Recall that $a: X\to C=\Alb(X)$ is the Albanese map where $C$ is an elliptic curve. Fix $P\in \Pic^0(C)\simeq \Pic^0(X)$. By Corollary \ref{m4}(2), $|3L+a^*P|$ separates two general points on two different fibers of a. 

Fix a general fiber $F$ of $a$, which is a surface with $K_F\equiv 0$ and $H=L|_F$. 
If $|2H|$ gives a birational map, then by Corollary \ref{m4}(3), $|3L+a^*P|$ separates two general points on $F$; if $|2H|$ does not give a birational map, then $|3L+a^*P|$ also separates two general points on $F$ by Lemma \ref{key lemma}, Propositions \ref{separate K3}, \ref{separate Enriques}, \ref{separate ab}, and \ref{key lemma biell} unless $F$ is an abelian surface  with $H^2=2$. 

Now we consider the case that $F$ is an abelian surface   with $H^2=2$.
In this case the assumption of Lemma \ref{L-F nef} is satisfied. Hence
$a_*\OO_X(3L)\otimes P$ is generated by global sections. 
Hence 
\begin{align*}
H^0(X, \OO_X(3L)\otimes a^*P){}&\simeq H^0(C, a_*\OO_X(3L)\otimes P)\\
{}&\to a_*\OO_X(3L)\otimes P\otimes \CC(z)\simeq H^0(F, \OO_F(3H))
\end{align*}
is surjective where $z=a(F)\in C$. Since  $|3H|$ gives a birational map on $F$, $|3L+a^*P|$ separates two general points on $F$.



Hence we proved that $|3L+a^*P|$ separates two general points on $X$, and whence gives a birational map.
\end{proof}

Finally, Theorem \ref{main} follows from  Corollary \ref{m4}(1), Theorems \ref{fiber one} and \ref{fiber two} directly.

\section{Irregular $4$-folds with $K\equiv 0$}

In this section, we prove Theorem \ref{main2}. We prove the following general theorem.

\begin{thm}\label{main4}
Let $X$ be a smooth projective variety with $K_X\equiv 0$, $a : X \to  A=\Alb(X)$ the Albanese map, and $L$ an ample divisor on $X$. If $\dim X-\dim A\leq 3$, then
$|mL+P|$ gives a birational map for all $m\geq 5$ and all $P\in \Pic^0(X)$.
\end{thm}
\begin{proof}
Let $X$ be a smooth projective variety  with $K_X\equiv 0$ and $L$ an ample divisor on $X$. Take $a: X\to A=\Alb(X)$ be the Albanese map and fix $P\in \Pic^0(A)$. By Theorem \ref{main3}, we may assume that $\dim A=\dim X-3=1$. By Theorem \ref{main3}(1), it suffices to prove that $|5L+a^*P|$ gives a birational map. Fix a general fiber $F$ of $a$, which is a smooth projective $3$-fold with $K_X\equiv 0$, and denote $H:=L|_F$. By Corollary \ref{separate2}(2), it suffices to prove that $|5L+a^*P|$ separates two general points on $F$. By Corollary \ref{separate2}(4), we may assume that $|4H|$ does not give a birational map on $F$. Hence by Theorem \ref{main}, $q(F)=0$, which implies that $$h^3(\OO_F)=h^0(\OO_F)+h^2(\OO_F)>0$$ and hence $K_F\sim 0$. In other words,  $(F, H)$ is a polarized Calabi--Yau $3$-fold such that $|4H|$ does not give a birational map. As classified by Oguiso \cite[Theorem (1.1)]{O}, there are only two cases:
\begin{enumerate}
\item[(I)] $F=(10)\subset \bP(1,1,1,2,5)$  with $H=\OO_F(1)$, or
\item[(II)] $h^0(F, \OO_F(H))=1$.
\end{enumerate}

For the second case, the assumption of Lemma \ref{L-F nef} is satisfied. Hence
$a_*\OO_X(5L)\otimes P$ is generated by global sections and therefore 
\begin{align*}
H^0(X, \OO_X(5L)\otimes a^*P){}&\simeq H^0(A, a_*\OO_X(5L)\otimes P)\\
{}&\to a_*\OO_X(5L)\otimes P\otimes \CC(z)\simeq H^0(F, \OO_F(5H))
\end{align*}
is surjective where $z=a(F)\in A$. Since  $|5H|$ gives a birational map on $F$, $|5L+a^*P|$ separates two general points on $F$.

For the first case, it is easy to see that 
\begin{enumerate}
\item $|4H|$  gives a generically finite morphism $\phi_{|4H|}$ of degree $2$;
\item every section in 
$$H^0(F, \OO_F(5H))\backslash \Image (m_{1423})$$
separates  two points lying in one general fiber of $\phi_{|4H|}$,
here $m_{1423}$ is the multiplication map of sections
\begin{align*}
m_{1423}:H^0(F, \OO_F(H)){}&\otimes H^0(F, \OO_F(4H))\\\oplus H^0(F, {}&\OO_F(2H))\otimes H^0(F, \OO_F(3H)) \\ {}& \to H^0(F, \OO_F(5H)).
\end{align*}
\end{enumerate}
It  is enough to show that $|5L+a^*P|$ separates   two   points lying in one general fiber of $\phi_{|4H|}$.

Here we use the idea of proof of Lemma \ref{key lemma} or Proposition \ref{key lemma biell} again. Since $q(F)=0$ in this situation, the assumption that $A$ is an elliptic curve is no longer needed as in the proof of Lemma \ref{key lemma} or Proposition \ref{key lemma biell}, which was kindly pointed out to the author by the referee.

Consider the exact sequence
$$
0\to \FF \to a_*\OO_X(5L)\to \QQ\to 0
$$
where $\FF$ is the image of the multiplication morphism
\begin{align}\label{morphism 1423}
a_*\OO_X(L)\otimes a_*\OO_X(4L)\oplus a_*\OO_X(2L)\otimes a_*\OO_X(3L)\to a_*\OO_X(5L).
\end{align}
Denote  $z=a(F)\in A$ a general point. Note that by definition, 
$
\QQ\otimes \CC(z)
$
is the cokernel of 
$m_{1423}$,
which is non-zero since $H^0(F, \OO_F(5H))$ gives a birational map but $\Image(m_{1423})$ does not.
Hence $\QQ\neq 0$.

We claim that $\QQ$ and $\FF$ are ample locally free sheaves.
By Subsection \ref{subsection 2.1}, there is an \'etale covering $\pi: B\to A$ such that $X\times_A B\simeq F\times B$,
we have the following commutative diagram:
$$\xymatrix{
F\times B\ar[d]^\sigma \ar[r]^{p_B} & B   \ar[d]^\pi \\
 X   \ar[r]^a  & A.}
$$
Note that, by  $q(F)=0$,  $\sigma^*\OO_X(L)=\OO_F(L_1)\boxtimes \OO_B(L_2)$ where $L_1$ is an ample divisor on $F$ and $L_2$ an ample divisor on $B$.
Hence for any interger $m$, $$\pi^*a_*\OO_X(mL)={p_B}_*\sigma^*\OO_X(mL)=H^0(F, \OO_F(mL_1))\otimes \OO_B(mL_2).$$
Hence after pull-back by $\pi$, map (\ref{morphism 1423}) turns out to be a map between direct sums of $\OO_B(5L_2)$. Hence $\pi^*\QQ$ and $\pi^*\FF$ are also direct sums of $\OO_B(5L_2)$, which are ample locally free sheaves on $B$. Since $\pi$ is \'etale, $\QQ$ and $\FF$ are ample locally free sheaves on $A$, in particular, they are $IT^0$.

By Lemma \ref{lemma 1}, $H^0(A, \QQ \otimes P)\neq 0$. Fix a non-zero section $\sigma \in H^0(A, \QQ \otimes P)$.
Since $z$ is general and $\QQ\otimes P$ is locally free, we may assume that 
$\sigma$ is not zero along $z$.
Since $\FF$  is $IT^0$, by taking cohomology,  it follows that 
\begin{align*}
H^0(X, \OO_X(5L)\otimes a^*P)\simeq H^0(A, a_*\OO_X(5L)\otimes P)\to H^0(A, \QQ\otimes P)
\end{align*}
is surjective. 
Hence $\sigma$ lifts to $\bar{\sigma}\in H^0(X, \OO_X(5L)\otimes a^*P)$.
Note that,  $\sigma$ is not zero along $z$ by construction, we have
$$
0\neq {\sigma}(z)\in \QQ\otimes P\otimes \CC(z)
$$
Hence $\bar{\sigma}(z)\in H^0(F, \OO_F(5H))\backslash \Image(m_{1423})$, and separates  two   points lying in one general fiber of $\phi_{|4H|}$ by assumption. Hence  $|5L+a^*P|$ separates   two   points lying in one general fiber of $\phi_{|4H|}$.

We complete the proof.
\end{proof}

\end{document}